\documentclass[a4paper,11pt]{article}
\usepackage[utf8]{inputenc}
\usepackage[T1]{fontenc}

\usepackage{amsthm}
\usepackage{tikz}
\usepackage{mathrsfs,amssymb,amsfonts}  
\usepackage{enumitem}
\usepackage{fullpage}

\usepackage{thmtools}
\usepackage{mathtools, comment}
\usepackage{amssymb}
\usepackage[nomath]{lmodern}

\usepackage{graphicx}
\usepackage{pgf,tikz,tkz-graph,subcaption}
\usetikzlibrary{arrows,shapes}
\usetikzlibrary{decorations.pathreplacing}
\usepackage{tkz-berge}
\usepackage{enumitem}
\usepackage[normalem]{ulem}
\usepackage{hyperref}
\hypersetup{colorlinks = true, linkcolor = blue, citecolor = blue, urlcolor = blue}

\newcommand*{\abs}[1]{\lvert #1\rvert}

\allowdisplaybreaks

\usepackage[margin=1in]{geometry}
\newtheorem{defi}{Definition}
\newtheorem{conj}[defi]{Conjecture}
\newtheorem{cor}[defi]{Corollary}
\newtheorem{thr}[defi]{Theorem}
\newtheorem*{theorem*}{Theorem}

\newtheorem{prop}[defi]{Proposition}

\newtheorem{question}[defi]{Question}
\newtheorem{remark}[defi]{Remark}
\newtheorem{claim}[defi]{Claim}
\newcommand*{\myproofname}{Proof}
\newenvironment{claimproof}[1][\myproofname]{\begin{proof}[#1]}{\end{proof}}

\def\F{\mathcal{F}}
\newcommand{\diam}{diam}
\newcommand{\rad}{rad}

\global\long\def\sc#1{\textcolor{blue}{\textbf{[SC comments:} #1\textbf{]}}}

\title{Set systems without a simplex, Helly hypergraphs and union-efficient families}
\author{Stijn Cambie\thanks{Extremal Combinatorics and Probability Group (ECOPRO), Institute for Basic Science (IBS), Daejeon, South Korea, supported by the Institute for Basic Science (IBS-R029-C4),
E-mail: {\tt stijn.cambie@hotmail.com, nikasalia@yahoo.com}} \and Nika Salia\footnotemark[1] 
}

\date{}

\begin{document}

\maketitle

\begin{abstract}

We present equivalent formulations for concepts related to set families for which every subfamily with empty intersection has a bounded sub-collection with empty intersection. Hereby, we summarize the progress on the related questions about the maximum size of such families. 

In this work we solve a boundary case of a problem of Tuza for non-trivial $q$-Helly families, by applying Karamata's inequality and determining the minimum size of a $2$-self-centered graph for which the common neighborhood of every pair of vertices contains a clique of size $q-2$.

\end{abstract}

\section{Introduction}

At first, in Subsection~\ref{subsec:introconcepts} we introduce new and existing concepts. 
In Subsection~\ref{subsec:relations} we show how these concepts are related and give a glimpse of some related problems.
In Subsection~\ref{subsec:summary}, we summarize the content of the paper.

Our notation follow~\cite{FT18}. 
Thus $[n]$ denotes $n$ element set $=\{1,2,\ldots,n\}$.
A subset of the power set of $[n]$, will be called a family $\F \subseteq 2^{[n]}.$
In the uniform case, every set of $\F$ is a set of size $k$, i.e. $\F \subseteq \binom{[n]}{k}.$
Such a family is  sometimes called a set system or $k$-uniform hypergraph, but in this work, we use the term family.
We use the following notation for the family of complements $\F^c= \{ A^c \mid A \in F\}$ where $A^c$ denotes $[n] \backslash A$.
The order and the size of a graph $G=(V,E)$ will be denoted with $\abs{V(G)}$ and $\abs{E(G)}$ respectively. The subgraph of $G$ induced by a set $A$ will be denoted by$~G[A].$

\subsection{Introduction of the concepts}\label{subsec:introconcepts}



\underline{\textbf{Union-efficient families}}

When the twin-free graph having the largest order for a given number of maximal independent sets was characterized in the work of~\cite{CW22+}, the notion of union-efficient families naturally appeared.

\begin{defi}\label{defi:unionefficient}
    For  fixed integers $n$ and $m$, we call a family $\F \subseteq 2^{[n]}$ \textbf{union-efficient} if for every subfamily $\{A_1, A_2, \ldots, A_{m}\} \subseteq \F$ for which $\cup_{i \in [m]} A_i = [n]$, there are two indices $i,j \in [m]$ for which $A_i \cup A_j =[n].$
\end{defi}

At first sight, the notion of union-efficient seemed to be a new concept, but by considering the complement, it is related to existing concepts.
On the other hand, the terminology of union-efficient families is more in line with existing basic terminology in extremal set theory, see e.g.~\cite{FT18}.
As such, it is also natural to consider the following variants.

\begin{defi}\label{defi:efficient2*}
    For  fixed integers $n$, $m$ and $q$, a family $\F \subseteq 2^{[n]}$ is \textbf{union-$q$-efficient} if for every subfamily $\{A_1, A_2, \ldots, A_{m}\} \subseteq \F$ for which $\cup_{i \in [m]} A_i = [n]$, there is a subset $I \subseteq [m]$ of cardinality $q$ for which $\cup_{i \in I} A_i = [n].$

    A family $\F \subseteq 2^{[n]}$ is \textbf{intersection-$q$-efficient} if for every subfamily $\{A_1, A_2, \ldots, A_{m}\} \subseteq \F$ for which $\cap_{i \in [m]} A_i = \emptyset$, there is a subset $I \subseteq [m]$ of cardinality $q$ for which $\cap_{i \in I} A_i = \emptyset.$
\end{defi}

Note that, union-efficient denotes union-$2$-efficient.
Trivial families play an important role in extremal set-theoretic problems.
For problems about the unions/ intersections of sets, a family $\F$ is called trivial if $\cup_{A \in \F} A \not= [n]$ resp. $\cap_{A \in \F} A \not= \emptyset$ and non-trivial if $\cup_{A \in \F} A = [n]$ respectively $\cap_{A \in \F} A = \emptyset.$

\underline{\textbf{Helly hypergraphs and simplices}}

There are multiple ways to define Helly families and simplices. We follow~\cite{Mulder83}, thus at first we introduce  $q$-linked families and then we define Helly families.

\begin{defi}
    A family $\F$ is \textbf{$q$-linked} if the intersection of any $q$ sets in $\F$ is non-empty. That is, $\forall A_1, A_2 , \ldots, A_q \in \F$, $A_1 \cap A_2 \cap \ldots \cap A_q \not= \emptyset.$
\end{defi}

Helly's celebrated theorem on convex sets states, a finite collection of $n$ convex subsets of $\mathbb R^d$ has a non-empty intersection if every $d+1$ subsets have a non-empty intersection. This inspired the following notion for families of sets.

    

\begin{defi}
    A family $\F$ satisfies the \textbf{Helly property} if every $2$-linked (pairwise intersecting) subfamily $\F'$ of $\F$ has non-empty intersection ($\cap_{F\in \F}F\neq \emptyset$).

    A family $\F$ satisfies the \textbf{$q$-Helly property} if every $q$-linked subfamily $\F'$ of $\F$ has non-empty intersection ($\cap_{F\in \F'}F\neq \emptyset$). 
\end{defi}

Here we present another important concept.

\begin{defi}
    A \textbf{$q$-simplex} is a family with $q+1$ sets $\{A_1, A_2, \ldots ,A_{q+1}\}$ such that their intersection is the empty set, but the intersection of any $q$ of them is not empty.
\end{defi}

\subsection{Relations between notions and with other problems}\label{subsec:relations}

The following theorem shows important connections between the concepts introduced in the previous subsection. 

\begin{thr}\label{thm:connection}
    For a family $\F \subseteq 2^{[n]}$, the following statements are equivalent.
    \begin{enumerate}[label=(\roman*)]
    \item \label{itm:1} $\F$ does not contain an $r$-simplex for any $r \ge q$
    \item \label{itm:2} $\F$ is $q$-Helly
    \item \label{itm:3} $\F$ is intersection-$q$-efficient
    \item \label{itm:4} $\F^c$ is union-$q$-efficient
    \end{enumerate}
\end{thr}

\begin{proof}
    We prove the equivalences of \ref{itm:1}, \ref{itm:2} and \ref{itm:3} in the cyclic order, after we show the equivalence of \ref{itm:3} and \ref{itm:4}.
    
    \begin{itemize}
        \item[ \ref{itm:1} $\Rightarrow $ \ref{itm:2}] We prove the contraposition $\neg$ \ref{itm:2} $\Rightarrow \neg$ \ref{itm:1}. Let $\F'$ be a minimal $q$-linked subfamily of $\F$ with empty intersection. Suppose $\F'$ contains $r+1$ sets.
        Then any subfamily of $\F'$ with $r$ sets would also be $q$-linked and thus would not have empty intersection (otherwise $\F'$ was not minimal).
        Hence $\F'$ is a $r-$simplex. The condition $r \ge q$ holds since the intersection of $q$ sets is non-empty, while the intersection of the $r+1$ sets is empty.
        \item[\ref{itm:2} $\Rightarrow$ \ref{itm:3} ] 
        Suppose $\F$ is a $q$-Helly family and let $\F'=\{A_1, \ldots, A_r\} \subset \F$ be any subfamily whose intersection is the empty set.
        Since $\F$ is $q$-Helly, we know $\F'$ is not $q$-linked, so there are $q$ sets in $\F'$ with empty intersection.
        Since $\F'$ was taken arbitrarily, we know that $\F$ is intersection-$q$-efficient.
        \item[\ref{itm:3} $\Rightarrow$ \ref{itm:1}]
        Proving the contraposition $\neg$ \ref{itm:1} $\Rightarrow \neg$ \ref{itm:3} is immediate, since an $r$-simplex with $r\ge q$ contains $r+1$ sets with empty intersection for which no $r$ and hence no $q$ have empty intersection.
        \item [\ref{itm:3} $\Leftrightarrow$ \ref{itm:4}]
        Since the intersection of $r$ sets in $\F$ is the empty set if and only if the union of the complements of the $r$ sets (which belong to $\F^c$) is $[n],$ this equivalence is immediate from Definition~\ref{defi:efficient2*}.  \qedhere
    \end{itemize}
\end{proof}

Another equivalent form was established by Berge and Duchet~\cite[Thr.1]{BD75}.
\begin{thr}[\cite{BD75}]\label{thr:BD75}
    A family $\F \subseteq 2^{[n]}$ is $q$-Helly if and only if for every $A \subseteq [n]$ such that $\abs{A}=q+1,$
    $$ \bigcap_{B \in \F \colon \abs{A \cap B} \ge q} B \not= \emptyset.$$
\end{thr}

Asking the question about maximal uniform families satisfying one of the presented conditions turns out to be interesting and challenging.
One of the simplest non-trivial cases turns out to be equivalent to open hypergraph-Turán-problems, originating from the work of Tur\'{a}n~\cite{Tur41} (see~\cite{Keevash11} for a survey).
In contrast to simpler graph cases, there are plausibly many extremal hypergraphs~\cite{Kost82} and there is no stability in general~\cite{LM22} for hypergraph-Turán-problems. The equivalence was observed before in~\cite{BD79, Mulder83}, but 
for completeness, we prove it here from scratch, note that it also can be seen as a corollary of Theorem~\ref{thr:BD75}.

\begin{prop}\label{prop:equi_HT}
    Let $\F$ be a subset of $\binom{[n]}{3}$, then $\F$ is intersection-$3$-efficient if and only if $\F$ is a $3$-uniform hypergraph without a subgraph isomorphic to $K_4^{(3)}$, a four vertex $3$-uniform complete hypergraph.
\end{prop}

\begin{proof}
    If $\F$ is an intersection-$3$-efficient family then it is $K_4^{(3)}$-free by the definition.
    
    Let $\F$ be a $3$-uniform $K_4^{(3)}$-free hypergraph and let $\cap_{i\in[m]}A_i=\emptyset$ for $\{A_1, A_2, \ldots, A_m\} \subseteq \F$ and  an integer $m\geq 3$. 
    If there are two sets $A_i$ and $A_j$ such that $\abs{A_i\cap A_j}\leq 1$,  then there is  a set $A_{\ell}$ such that $\ell \in [m]$ and $A_{\ell}$ is disjoint from $A_i\cap A_j$ since $\cap_{i\in[m]}A_i=\emptyset$. Hence $\F$ is an intersection-$3$-efficient family.
    If for all $i,j\in[m]$ we have $\abs{A_i\cap A_j}=2$ then, let $A_1=\{a_1,a_2,b_1\}$ and $A_2=\{a_1,a_2,b_2\}$. Since the intersection $\cap_{i\in[m]}A_i$ is an empty set, there is a set not containing $a_1$ and there is another set not containing $a_2$. Those two sets are $\{a_2,b_1,b_2\}$ and $\{a_1,b_1,b_2\}$,  since all pairwise intersections have size two. A contradiction since  $A_1$, $A_2$, $\{a_2,b_1,b_2\}$ and $\{a_1,b_1,b_2\}$ is a copy of $K_4^{(3)}$.
    \end{proof}
        
We shortly mention that this connection was some additional motivation to work on a boundary case of a problem of Tuza. 
The largest $\F \subseteq \binom{[n]}{k}$ that are intersection-$3$-efficient for values $\frac{2n}{3}\ge k>3$ turn out to be the trivial families (Theorem~\ref{thr:Helly=>StarExtr}).
The latter result does not give insight in the most basic hypergraph-Turán-problem.
Since the largest intersection-$3$-efficient family $\F \subseteq \binom{[n]}{3}$ is non-trivial, one may hope that determining the largest non-trivial intersection-$3$-efficient $\F \subseteq \binom{[n]}{k}$ where $\frac{2n}{3}\ge k>3$ is more interesting.
That question had been posed before by Tuza~\cite{Tuza93} (see question~\ref{ques:Tuza}) and is still widely open. 

The set-up of Helly graphs has also been connected with the transversal number, see e.g.~\cite{Gy78}, and has been studied for Sperner families~\cite{BD83}.
Also as an analog of $t$-intersecting families (for $t=2$), bi-Helly families have been considered~\cite{Tuza00}.
A related extension is the notion of a special simplex~\cite{FF87}.
A special $q$-dimensional simplex is a family $\{A_1,A_2, \ldots, A_{q+1}\}$ such that there exists a set $C=\{x_1,x_2,\ldots, x_{q+1}\}$ for which  $A_i \cap C= C \backslash x_i$ for every $1 \le i \le q+1$ and all $A_i \backslash C$ are disjoint.
tIn this work, we do not focus on these related versions.

Finally, we observe that there are more notions that are similar in flavor. 
A family $\F$ of sets has the $(p, q)$ property if among any $p$ sets of $\F$ some $q$ have
a nonempty intersection. This property was invented by Hadwiger and Debruner~\cite{HD57}, where they extend the result of Helly on convex sets in $\mathbb R^k.$
While the definition is stated for set families, it did not get attention in this more general framework.
A few exceptions are $p=q$ and $q=2$, stated in a different way.
A family satisfies the $(p,p)$ property precisely if it is $p$-intersecting.
When $q=2$, 
a family $\F$ of sets has the $(p, 2)$ property if $\F$ contains no $p$ disjoint sets.
Hence Kleitman~\cite{Kleitman68} studied this case already.
More general, one can ask about the maximum size of a family $\F \subseteq 2^{[n]}$ which has the $(p,q)$ property for $p> q>2.$

\begin{question}
   Given $p>q>2.$ What is the maximum size of a family $\F \subseteq 2^{[n]}$ which has the $(p,q)$ property. 
\end{question}

\subsection{Overview of content}\label{subsec:summary}

In Section~\ref{sec:overview}, we summarize the progress on maximum Helly families and families without simplices.
We end with a problem of Tuza~\cite{Tuza93} about the maximum size of a non-trivial uniform $q$-Helly family $\F \subseteq \binom{[n]}{k}$, which we solve in a boundary case $n=k\frac{q}{q-1}$.
Equivalently, we determine the largest non-trivial union-$q$-efficient family $\F \subseteq \binom{[n]}{k}$ where $n=qk.$
In Section~\ref{sec:q=2} we solve the case $k=2$ separately.
In Section~\ref{sec:largestsize_q} this is done for $k \ge 3$. 
Now we present the main result of this work.

\begin{thr}\label{thr:main}
    Let $n=qk$ where $q, k \ge 2$, $(q,k)\not=(2,2)$ and let $\F \subseteq \binom{[n]}{k}$ be a non-trivial union-$q$-efficient family. 
    Then $\lvert \F \rvert \le \binom{n-q}{k}+q.$ 
    Furthermore, the extremal family is unique up to isomorphism.
\end{thr}

The proof uses some results on self-centered (graphs with the radius equal to the diameter) graphs for which the common neighborhood of any $2$ vertices contains a $K_{q-2}.$
For $q=3,$ this shows that adding the triangle-property as a condition to the result of Buckley~\cite{Buckley79}, implies that the lower bound on the size of a $2$-self-centered graph (graph with $\rad(G)=\diam(G)=2$) goes up from $2n-5$ to $2n-3.$
Finally, we give some concluding remarks in Section~\ref{sec:conc}.

\section{Overview of results on maximum Helly families and families without simplices}\label{sec:overview}

In this section, we summarize some important theorems connected to the previously presented concepts.

\subsection{Non-uniform families of maximal size}

Milner, as mentioned in~\cite{Erdos71}, proved the following theorem.

\begin{thr}[Milner]
    Let $\F \subseteq 2^{[n]}$ be a family without $3$-simplex.
    Then $\abs{\F}\le 2^{n-1}+n.$
\end{thr}

The family $\F=\{ A \mid 1\in A \subseteq [n]\} \cup \binom{[n]}{ \le 1}$ shows sharpness of this theorem. Note that by Theorem~\ref{thm:connection} we have the family $\F=2^{[n-1]} \cup \binom{[n]}{ \le n-1}$ is an extremal union-efficient family.

Bollob\'as and Duchet~\cite[Cor.~3]{BD79}, with the uniqueness statement proven in~\cite[Thr.~2]{BD83},
and Mulder~\cite[Thr.~2]{Mulder83} generalized the above theorem for $q$-Helly families.

\begin{thr}[\cite{BD79},\cite{Mulder83}]
    Let $\F \subseteq 2^{[n]}$ be a $q$-Helly family.
    Then $\abs{\F}\le 2^{n-1}+\binom{n-1}{\ge n-q}.$
    Furthermore, equality holds if and only if for some $i \in [n]$ $\F$ equals
$\{ A \mid i \in A \subseteq [n]\} \cup \binom{[n]}{ \le q-1}.$\end{thr}

It took $25$ years  to prove that the same bound holds for families without a $q$-simplex.
Keevash and Mubayi~\cite{KM10} proved the following theorem.

\begin{thr}[\cite{KM10}]
    Let $\F \subseteq 2^{[n]}$ be a family without $q$-simplex. 
    Then $\abs{\F}\le 2^{n-1}+\binom{n-1}{\ge n-q}.$
\end{thr}

\subsection{Maximum uniform families without $r$-simplices are typically trivial}\label{subsec:overviewtrivialcase}

In 1974, inspired by a problem of Erd\H{o}s~\cite{Erdos71} and the Erd\H{o}s-Ko-Rado theorem~\cite{EKR61}, Chv\'{a}tal~\cite{Chvatal74} conjectured that maximum uniform families without $r$-simplices are typically trivial.
This conjecture is known as the Erd\H os-Chv\'{a}tal Simplex Conjecture.

\begin{conj}[\cite{Chvatal74}]\label{conj:Chvatal}
    Let $k > r \ge 1$ and $n \ge \frac{r+1}{r} k.$
    A family $\F \subseteq \binom{[n]}k$ without a $r$-simplex contains at most $\binom{n-1}{k-1}$ sets.
\end{conj}


The theorem of Erd\H{o}s-Ko-Rado~\cite{EKR61} can be formulated in this form.

\begin{thr}[\cite{EKR61}]
    A family $\F \subseteq \binom{[n]}k$ without a $1$-simplex contains at most $\binom{n-1}{k-1}$ sets.
\end{thr}

If one forbids all simplices of size at least $q,$ instead of only the simplices of size $q$, the problem is easier.
In \cite[Thr.~1]{Mulder83}, Mulder proved the upper bound for $q$-Helly families.
\begin{thr}[\cite{Mulder83}]\label{thr:Helly=>StarExtr}
    Let $\F \subseteq \binom{[n]}k$ be a $q$-Helly family, where $k>q$.
    Then $\abs{\F}\le \binom{n-1}{k-1}$ and equality is attained only if $\F$ is a trivial family
\end{thr}

Chv\'{a}tal~\cite{Chvatal74}  proved the $k=r+1$ case of Conjecture~\ref{conj:Chvatal}. The case $r=2$, which was the initial problem of Erd\H{o}s, was proven only $30$ years later by Mubayi and Verstra\"ete~\cite{MV05}.
Here they considered hypergraphs without a non-trivial intersecting sub(hyper)graph of size $q+1$.
Here it is interesting to note that in that set-up~\cite[Thr.3]{MV05}, for $k=3$, if the size of the non-trivial intersecting family has a large size $q+1 \ge 11$ the star is not extremal.
Liu~\cite{Liu22+}~proved that the star is still extremal if $k>3$ and $n$ is sufficiently large.

Conjecture~\ref{conj:Chvatal} was proven to be true provided that $n$ is sufficiently large in terms of $k,r$ by Frankl and Furedi~\cite{FF87} and in terms of only $r$ by Keller and Lifshitz~\cite{KL21}.
Currier~\cite{Currier21} proved it when $n \ge 2k-r+2.$
For more insights on the history of Conjecture~\ref{conj:Chvatal} and related problems, we refer the reader to~\cite[Sec.~6.4]{liu2022extremal}.

\subsection{Maximum non-trivial Helly families}

Since the extremal families are the trivial ones (Theorem~\ref{thr:Helly=>StarExtr}), it is natural to wonder what happens with non-trivial families.

\begin{question}[\cite{Tuza93}]\label{ques:Tuza}
    What is the maximum possible size of a non-trivial $k$-uniform $q$-Helly family ${\F \subseteq \binom{[n]}{k}}$?
\end{question}

Tuza~\cite[Thr.~1.5]{Tuza94} solved this question for $q=2$ provided that $n>>k.$

\begin{thr}[\cite{Tuza94}]
    For $n$ sufficiently large in terms of $k$, a non-trivial Helly family $\F \subseteq \binom{[n]}{k}$ satisfies 
    $$\abs{\F} \le \binom{n-k-1}{k-1}+\binom{n-2}{k-2}+1.$$
    Furthermore, the extremal family is unique (up to isomorphism).
\end{thr}

\section{Maximum size graphs for which every spanning subgraph has a perfect matching}\label{sec:q=2}


In this section, we prove the $k=2$ case of Theorem~\ref{thr:main} (the $k\ge3$ case will be handled with a different strategy in Section~\ref{sec:largestsize_q}.
This case can be stated completely with basic terminology in graph theory; spanning subgraphs and perfect matchings.
We first prove the case where the graph is connected.

\begin{thr}\label{thm:spanning_matching}
    Let $n=2q$ and $G$ be a connected graph of order $n$ for which every spanning (not necessarily connected) subgraph $H$ has a perfect matching.
    Then the maximum size of $G$ equals $1$ or $4$ if $q \in \{1,2\},$
    or $\binom{q+1}2$ if $q \ge 3.$ Furthermore, the extremal graph is unique.
\end{thr}

\begin{proof}
    For $q=1$ we are trivially done.
    For $q=2$ the graph does not contain a vertex adjacent to the rest of the vertices, therefore the maximum degree is two and we have at most four edges. Note that the bound is tight since a cycle of length four has the desired properties.
    
    Let us assume $q \ge 3$.
    Let $M=\{u_iv_i\}_{1 \le i \le q}$ be a perfect matching of $G$ (which exists by choosing $H=G$).
    We claim that for every edge of $M$, one of the two vertices is a leaf.
    \begin{claim}
        For every edge $u_iv_i \in M,$ one of its end-vertices $u_i, v_i$ is a leaf.
    \end{claim}
    \begin{claimproof}
       Without loss of generality, we may assume $i=1$. 
       Since $G$ is a connected graph, $u_1$ or $v_1$ is adjacent to  a vertex of $G$ different from $v_1$ and $u_1$.  Without loss of generality, we may assume that $u_1$ is adjacent to $u_2.$
        Now we will prove that $v_1$ is a leaf.
        The vertex $v_1$ is not adjacent to vertex $u_2$. Since otherwise the edge set $\{u_2v_1, u_2v_2, u_2u_1\} \cup \{u_iv_i\}_{3 \le i \le q}$ spans a graph $G$ without a perfect matching.
        The vertex $v_1$ is not adjacent to any vertex with an index greater than two. Since otherwise, let us assume  without loss of generality that $v_1$ is adjacent with  $v_3$, then  $v_1v_3 \in E(G),$ then $\{u_2v_2, u_2u_1, v_1v_3, v_3u_3\} \cup \{u_iv_i\}_{4 \le i \le q}$ spans   graph $G$ without a perfect matching.
        Finally, if $v_1$ is adjacent to the vertex $v_2$ then by the previous argument  none of the vertices $v_1,v_2,u_1,u_2$ vertices is adjacent to  a vertex with index larger than $2$, a contradiction since $G$ is a connected graph with more than four vertices.
    \end{claimproof}
    Since the graph has $q$ leaves, we note that $G$ is a subgraph of a clique $K_q$ with a pendent vertex for each vertex of $K_q$.
    The latter graph has size $\binom{q+1}{2}$ and every spanning subgraph contains a perfect matching.
\end{proof}

As a corollary, we derive Theorem~\ref{thr:main} for $k=2.$

\begin{cor}
    Let $n=2q$ and $\F \subseteq \binom{[n]}{2}$ be a non-trivial union-$q$-efficient family.
    Then $\abs{\F} \le \binom{q+1}{2}$ whenever $q \ge 3.$
\end{cor}

\begin{proof}
    Note that a family $\F \subseteq \binom{[n]}{2}$ corresponds with the edge-set of a graph $G$.
    Being union-$q$-efficient implies here that any spanning graph contains a perfect matching. Thus if $G$ is connected we are done by Theorem~\ref{thm:spanning_matching}.  If $G$ is  not connected then we are done by induction since the following inequalities hold
    $\binom{q_1+q_2+1}{2} > \binom{q_1+1}{2}+\binom{q_2+1}{2}, \binom{q_1+2}{2}>\binom{q_1+1}{2}+1$, $\binom{q_1+3}{2}>\binom{q_1+1}{2}+4$, $\binom{4+1}{2}>4+4$ and $\binom{3+1}{2}>4+1.$
\end{proof}

Remark that the maximum size of a family $\F \subseteq \binom{[n]}{n-2}$ without $q$-simplex is different for $n=2q\ge 10.$
Let $\F^c$ be the graph $K_{n-4}$ with four additional vertices connected to the same vertex of the $K_{n-4}$. This graph has $\binom{n-4}{2}+4>\binom{q+1}{2}$ edges,
while there are no $q+1$ edges spanning all vertices of the graph and thus in $\F$ there is no $q$-simplex.
This indicates the clear difference between forbidding a $q$-simplex and forbidding all $r$-simplices with $r \ge q.$

\section{Largest non-trivial union-$q$-efficient families}\label{sec:largestsize_q}

In this section, we prove Theorem~\ref{thr:main} which we restate for the convenience of the reader.

\begin{theorem*}
    Let $n=qk$ where $k \ge 3$ and let $\F \subseteq \binom{[n]}{k}$ be a non-trivial union-$q$-efficient family. 
    Then $\lvert \F \rvert \le \binom{n-q}{k}+q.$ 
    Furthermore, the extremal family is unique up to isomorphism.
\end{theorem*}

\begin{proof}
    The case $q=1$  trivially holds.  
    The case $q=2$ holds by an easier version of the proof for $q \ge 3.$ Thus we assume $q\geq 3.$
    Let $\F \subseteq \binom{[n]}{k}$ be a non-trivial union-$q$-efficient family as in the statement.
    For every $j \in [n],$ let $\F(j)=\{A \in \F \mid j \in A\}$ be the family of sets in $\F$ containing $j$ and $X(j)=\cup\{A \in \F \mid j \in A\} \subseteq [n]$ be the set of elements which are covered by $\F(j).$
    An important observation is the following.
    
    \begin{claim}\label{clm:2unionnot[n]}
       There is no index set $J$ with $J\subseteq [n]$ and $\abs{J}=q-1$ such that $\cup_{j \in J} X(j) =[n]$ 
    \end{claim}
    \begin{claimproof}
        Suppose by way of contradiction that there is index set $J$ with $J\subseteq [n]$ and $\abs{J}=q-1$ such that $\cup_{j \in J} X(j) =[n]$.  Then since $\F$ is union-$q$-efficient, there must be $q$ sets such that their union is $[n]$. Thus these sets must be disjoint since $n=qk$, but this is impossible since there will be at least two sets sharing an element from $J$ by the pigeonhole principle, a contradiction.
    \end{claimproof}
    Next, we prove an upper bound for the size of $\F(j)$.
    \begin{claim}\label{clm:uppbound}
        For every $j \in [n]$ we have 
        \[\abs{\F(j)}\leq \binom{\lvert X(j)\rvert -2}{k-1}+1.\]
    \end{claim}
    \begin{claimproof}
         The family $\F$ is union-$q$-efficient and non-trivial, thus  there are $q$ sets  $A_1, A_2, \ldots, A_q \in \F$ such that $\cup_{i\in [q]}A_i=[n]$.
         Thus since $n=qk$ and $\F$ is $k$-uniform the sets $A_i$, $i\in [q]$ are disjoint. 
        Without loss of generality, we may assume that $j \in A_1$. 
        By Claim~\ref{clm:2unionnot[n]} it is easy to note that  $A_2, A_3, \ldots, A_q \not \subseteq X(j)$.  Thus each $A_i$, $1<i\leq q$, contains a unique element of $[n]$ which is not an element of $X(j)$.  The family $\F(j) \backslash \{A_1\}$ does not covers all elements of $A_1$, since otherwise the family $\F(j) \backslash \{A_1\} \cup \left( \cup_{i=2}^{q} \{A_i\}\right)$ covers $[n]$, but there are no $q$ sets covering $[n]$ contradicting to the condition that $\F$ is union-$q$-efficient.   
        Hence $\F(j) \backslash A_1$ covers at most $\abs{X(j)}-1$ elements of $X(j)$, thus we have the desired inequality.
    \end{claimproof} 
    Let $G$ be the graph with vertex set $[n]$ for which $i,j \in E(G)$ if and only if there is no $A \in \F$ for which $\{i,j\} \subseteq A$, i.e $G$ is a complement of the $2-$shadow of $\F$.
    This graph satisfies the following properties.
    \begin{claim}\label{clm:propertiesG}
        For every two vertices $u,v$ in $G$, their common neighborhood $G[N(u)\cap N(v)]$ contains a clique on $q-2$ vertices.
        The minimum degree of $G$ is at least $q-1$ and the maximum degree is bounded by $(q-1)k.$
    \end{claim}
    \begin{claimproof}
    For every $j \in [n]$  we have $\lvert X(j) \rvert \le n-(q-1)$ by Claim~\ref{clm:2unionnot[n]}. Thus the degree of vertex $j$ in $G$ is at least $q-1$ i.e. $\delta(G) \ge q-1.$
    
    By Claim~\ref{clm:2unionnot[n]}, for every $j,j'\in [n]$ there is $i_1 \not \in X(j) \cup X(j')$, i.e., $j$ and $j'$ have a common neighbor in $G$.
    One can repeat this for $J_\ell=\{j,j',i_1,\ldots, i_{\ell}\}$ by taking an $i_{\ell+1} \not \in \cup_{\gamma \in J_{\ell}} X(\gamma)$ whenever $\ell \le q-3.$
    Thus $G[\{i_1,\ldots, i_{q-2}\}]$ is a clique and all vertices $\{i_1,\ldots, i_{q-2}\}$ are adjacent to both $j$ and $j'$.
    
    Finally, every $j \in [n]$ belongs to at least one $k$-set in $\F$ since $\F$ is non-trivial and thus $\Delta(G)\le n-k=(q-1)k.$
    \end{claimproof}

    Let $\{d_1,d_2, \ldots, d_n\}$ be the degree sequence of $G$.
    As Claim~\ref{clm:propertiesG} implies that $G$ satisfies $\diam(G)=\rad(G)=2$ and the property that every common neighborhood of $2$ vertices contains a $K_{q-2}$,
    by Theorem~\ref{thr:mge2n-3} (for $q=3$) and Theorem~\ref{thr:Kq-2inN} (when $q>3$) and the handshaking lemma, we know that $\sum_{i=1}^n d_i \ge (q-1)(2n-q).$
    If this inequality is strict, decrease some of the values with the constraint that all their values are still at least $q-1$ in such a way that the sum is $(q-1)(2n-q).$ Let the resulting sequence be $\{x_1,x_2,\ldots, x_n\}.$
    Note that the latter sequence is majorized by $\{\underbrace{q-1,\ldots,q-1}_{n-q},\underbrace{(q-1)k,\ldots,(q-1)k}_{q}\}.$
    That is, the sum of the largest $i$ elements in the latter sequence is at least the sum of the largest $i$ elements in the sequence $\{x_1,x_2,\ldots, x_n\},$ for every $1 \le i \le n$, with equality if $i=n.$
    Let $f \colon \mathbb R \to \mathbb R \colon x \to \binom{n-x-2}{k-1}+1.$
    Restricted to the interval $[q-1,(q-1)k]$, this is a strictly convex function.
    By Karamata's inequality~\cite{Karamata32} and the fact that $f$ is decreasing, we have
    $$ \sum f(d_i) \le \sum f(x_i) \le (n-q)f(q-1)+qf((q-1)k).$$
    By Claim~\ref{clm:uppbound}, where $\abs{X(j)}=n-d_j$ and double-counting (each set is counted $k$ times), we conclude that 
    \[
\abs{\F}=\sum_{j\in [n]}\frac{\abs{\F(j)}}{k} \le \frac{(n-q)f(q-1)+qf((q-1)k)}{k}
    =\frac{n+(n-q)\binom{n-q-1}{k-1}}{k}
    =q+\binom{n-q}{k}. 
    \]
    When equality is attained, there are $q$ elements (without loss of generality $n-q+1$ till $n$) belonging to a unique $k$-set (since $d_i=(q-1)k$) and there are at most $\binom{n-q}{k}$ other $k$-sets which do not contain any of these $q$ elements, so all of these $k$-sets need to be contained in $\F$.
    The first $q$ sets have to be different and if they are not disjoint, there is an element $j \in [n]$ for which $\abs{X(j)}\ge n-k+2$, which is a contradiction.
    Hence equality does occur if and only if there are $q$ elements belonging to a unique (disjoint) $k$-set, and all $k$-sets of the remaining $n-q$ elements belong to $\F.$
    Noting that this family is union-$q$-efficient is immediate since a union of sets from the family can only be equal to $[n]$ if the $q$ disjoint sets all belong to the family.
    An example of a maximum family $\F$ and the corresponding graph $G$ has been given in Figure~\ref{fig:Gsize2n-5} for $k=4, q=3$.
    Here every $4$-set within the light grey box belongs to $\F.$
\end{proof}

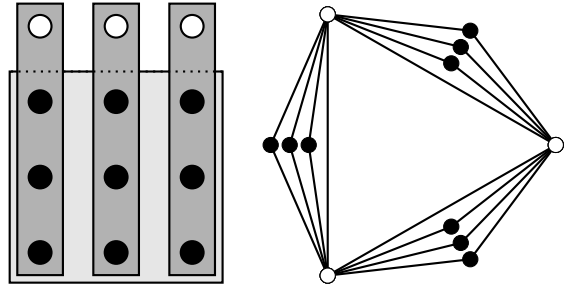
\begin{figure}[h]
    \centering
    \centering
\begin{tikzpicture}[thick]
\draw[fill=black!10!white] (1-0.4,0.6) rectangle (3+0.4,3.4);
\foreach \x in {1,2,3}{

\draw [fill=black!30!white] (\x-0.3,0.7) rectangle (\x+0.3,4.3);
\draw[fill=white] (\x,4) circle (0.15);
\foreach \y in {1,2,3}{
\draw[fill] (\x,\y) circle (0.15);
}

}
\draw[dotted] (3.4,3.4) -- (0.6,3.4);
\end{tikzpicture}
\quad
    \begin{tikzpicture}
    
    \foreach \x in {0,120,240}{\draw[] (\x:2) circle (0.1);
    \draw[thick](\x:2) -- (\x+120:2);
}

  \foreach \x in {60,180,300}{
  \foreach \t in {1.25,1.5,1.75}
  {
  \draw[fill] (\x:\t) circle (0.1);
  \draw[thick] (\x-60:2)--(\x:\t) -- (\x+60:2);
}
\foreach \x in {0,120,240}{\draw[fill=white] (\x:2) circle (0.1);
}
}   

\end{tikzpicture}
    \caption{Sketch of a maximum  union-$3$-efficient family $\binom{[12]}{4}$ and the associated graph $G$ with $\diam(G)=\rad(G)=2$ and the triangle property}
    \label{fig:Gsize2n-5}
\end{figure}

We remark that as was the case with $k=2$, the largest non-trivial $q$-simplex-free families can have a larger size than the largest non-trivial $q$-Helly family. E.g. when $n=qk$, let $\F^c=\{A \in \binom{[n]}{k} \mid \abs{ A \cap [q+2] } \le 1\}$. It has size $\binom{n-q}{k}+q\binom{n-q-2}{k-1}-\binom{n-q-2}{k-2}$ and is $q$-simplex-free.

\section{Minimum size of $2$-self-centered  graphs}\label{sec:minimumsizegraph}

Estimating the size of graphs (determining the minimum and the maximum) with some given parameters (mostly order and one other parameter) is a fundamental question in extremal combinatorics.
For example, finding the minimum/ maximum size of certain critical graphs with given order and diameter $2$ is challenging.
In~\cite{KE95, HY98,CF05} authors proved that the minimum size of a vertex-diameter-$2$-critical graph (a graph with diameter $2$ for which the diameter increases by deleting any of its vertices) is roughly $\frac 52 n.$

In~\cite{Buckley79} it was proven that a self-centered graph (a graph for which diameter and radius are equal, initially called equi-eccentric graph) with a diameter equal to $2$, has a size of at least $2n-5.$ The $2$-self-centered graphs with size equal to $2n-5$ have been characterized in~\cite{AA81}.
By observing that a graph with diameter $2$ has radius $2$ if and only if the maximum degree satisfies $\Delta< n-1$, the bound of $2n-5$ edges had been derived before by Erd\H{o}s and Renyi~\cite{ER62}.

A related question was solved in~\cite{BE76}, where the non-adjacent vertices have a minimum number of common neighbors.
We consider a similar question, where adjacent vertices have at least one common neighbor, i.e.
the graph $G$ has the triangle-property: every edge of $G$ is contained in a triangle.
This property has been studied before e.g. in~\cite{PR16} for $4$-regular graphs.
Since every graph with the triangle-property is the union of some triangles, for a connected graph with the triangle-property, the size $m$ satisfies $m \ge \frac 32(n-1).$
If $G$ is $2$-self-centered and has the triangle-property, we prove that its size is at least $2n-3$. 
We first prove such a result for when any $2$ vertices in their neighborhood share a $K_{q-2}$ for $q \ge 4$.


\subsection{The minimum size of a graph $G$ with a $K_{q-2}$ in the common neighborhood of every $2$ vertices}

In this subsection, we determine the minimum size of a $2$-self-centered graph $G$ with a copy of $K_{q-2}$ in the common neighborhood of every pair of vertices. More precisely, we prove the following theorem.

\begin{thr}\label{thr:Kq-2inN}
    Let $q\ge 4$ and $n\ge 3q$ be integers.
    Let $G$ be a graph for which $\rad(G)=2$ and such that for every $u,v \in G$, $G[N(u) \cap N(v)]$ contains a $K_{q-2}.$
    Then the number of edges of $G$ is at least $(q-1)n -\binom{q}{2}.$
\end{thr}

The lower bound in Theorem~\ref{thr:Kq-2inN} is sharp. Equality is attained by an $n$ vertex graph $G$, such that the vertex set of $G$ is partitioned into $q+1$ non-empty sets $A_0=[q],A_1,\dots,A_q$, where $G[A_0]$ is isomorphic to $K_q$, $A_i$ for $i\in [q]$ are independent sets and for each $i\in[q]$ every vertex of $A_i$ is adjacent to all vertices from $A_0\setminus \{i\}$.

Since every vertex belongs to a clique $K_{q}$, we have $\delta(G) \ge q-1.$
We prove the cases $\delta(G) = q-1$
and $\delta \ge q$ separately in the following two propositions.

First, we prove the statement in a more general form for $q\ge 4$ in the case that the minimum degree is exactly $q-1.$ 

\begin{prop}\label{prop:delta=Qcase}
    Let $q\ge 4$ and $n\ge 2q$ be integers.
    Let $G$ be a graph with $\rad(G)=2$, $\delta(G)=q-1$ such that for every $u,v \in G$, $\abs{N(u) \cap N(v)}\geq q-2$.
    Then the number of edges in $G$ is at least $ (q-1)n -\binom{q}{2}.$
\end{prop}
\begin{proof}
Let $v$ be a vertex of minimum degree of $G$, $\deg(v)=\delta(G)=q-1$ and let the neighborhood of $v$ be $N(v)=\{v_i:i\in [q-1]\}.$ 
Since for every $i \in [q-1]$, $v_i$ and $v$ share $q-2$ common neighbors $G[N[v]]$ is isomorphic to $K_{q}$.
Since $\rad(G)=2,$  for every vertex $v_i$, $i\in[q-1]$, there is a vertex $s_i\in V \backslash N[v]$ not adjacent to $v_i$. Since each vertex has at least $q-2$ neighbors in $N(v)$, vertices $s_i$ are distinct. Let us denote 
\[
R=\{w_1w_2 \in E(G) \colon \deg(w_1)=q-1, \abs{N(w_1)\cap N(v)}=q-2, w_1, w_2 \not\in N[v]\}.
\] 
Let $w_1w_2$ be an edge from $R$. 
We may assume $N(w_1)\cap N(v)=\{v_i:i\in[q-2]\}$.
Even more, since the common neighborhood of $w_1$, $w_2$ contains $q-2$ vertices, $w_2$ is adjacent with $v_1,\dots,v_{q-2}$. The vertex $w_2$ is also incident with $\{s_i: i\in [q-2]\}$, since $w_1$ and $s_i$ ($i\in [q-2]$) have at least $q-2$ vertices in the common neighborhood. Note that if $N(w_2)\cap N(v)=\{v_i:i\in[q-2]\}$, then $\deg(s_i)\geq q$ for $1\leq i\leq q-2$, since $w_2$ and $s_i$ share at least $q-2$ common neighbors.

Now we are ready to lower bound the number of edges. 
There are $\binom{q}{2}$ edges in $N[v]$.
For each vertex $u\in V\backslash N[v]$, there are at least $q-2$ edges from $u$ to $N(v)$, let $V_{1}\subseteq V\backslash N[v]$ be vertices with exactly $q-2$ neighbors in $N(v)$.
For each vertex $u\in V_1$ either $\deg(u)=q-1$ (and it is incident with exactly one edge in $R$) or $u$ is incident with at least $q$ edges from $E(G)\backslash R$, since if it is incident to at least one edge from $R$, then it is adjacent to $q-2$ vertices from $\cup_{i\in [q-1]}\{s_i\}$ none of which has degree $q-1$. Note that $q-2\geq 2$.
Thus for each vertex in $V\backslash N[v]$ there are either $q-1$ edges to $N(v)$, or there are $q-2$ edges to $N(v)$ and exactly one edge from $R$, or there are $q-2$ edges to $N(v)$ and at least two edges in $E[G[V\backslash N[v]]]\backslash R$.
If we associate these edges with the vertex, except the edges in $E[G[V\backslash N[v]]]\backslash R$ which are taken with a weight of a half, then every edge is counted (with weight) at most once and every vertex in $V \backslash N[V]$ is associated with a total weight of edges of at least $q-1$.
Hence there are at least $\binom{q}{2}+(n-q)(q-1)$ edges.
An example for which equality is attained is shown in Figure~\ref{fig:sketch}, where the edges in $R$ are presented in red.
\end{proof}

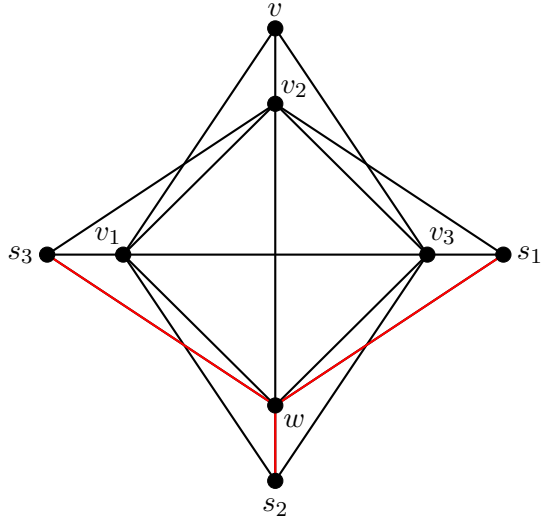
\begin{figure}[h]
    \begin{center}
    \begin{tikzpicture}
    {
    
    \foreach \x in {0,90,180,270}
    {
     \draw[thick] (\x:2) -- (\x+90:2);
     \draw[thick] (\x:3) -- (\x+90:2);
     \draw[thick] (\x:3) -- (\x:2);
     \draw[thick] (\x:3) -- (\x-90:2);
    }
    \draw[thick] (0:2) -- (180:2);
    \draw[thick] (90:2) -- (270:2);
    \foreach \x in {0,180,270}
    {
    \draw[thick, color=red] (\x:3) -- (-90:2);
    }
\foreach \x in {0,90,180,270}
    {
     \draw[fill] (\x:2) circle (0.1);
     \draw[fill] (\x:3) circle (0.1);
    }
    
   \node at (90:3.25) {$v$};

    \node at (-2.2,0.25) {$v_1$};
    \node at (0.25,2.2) {$v_2$};
    \node at (0.25,-2.2) {$w$};
    \node at (2.2,0.25) {$v_3$};
    \node at (180:3.35) {$s_3$};
    \node at (0:3.35) {$s_1$};
    \node at (270:3.35) {$s_2$};

    }
	\end{tikzpicture}\\
\end{center}
\caption{Extremal graph for Proposition~\ref{prop:delta=Qcase} with $n=8$ and $q=4$}\label{fig:sketch}
\end{figure}

\begin{remark}
    The condition  $\rad(G)=2$ is  necessary here. Without that condition, one can take $\frac{n-(q-2)}{2}$ copies of $K_q$ which pairwise intersect in a fixed copy of $K_{q-2}.$ The latter construction has a smaller size.
\end{remark}

Next, we prove the case where $\delta \ge q$. Here the statement is true without the constraint $\rad(G)=2.$

\begin{prop}\label{prop:deltaGeQcase}
    Let $q\ge 3$ and $n\ge 3q-4$ be integers.
    Let $G$ be a graph such that for every $u,v \in G$, $G[N(u) \cap N(v)]$ contains a $K_{q-2}$ and $\delta(G)\ge q.$
    Then the number of edges in $G$ is at least $ (q-1)n -\binom{q}{2}.$
\end{prop}

\begin{proof}
If $\delta(G) \ge 2(q-1),$ then by the handshaking lemma we have $\abs{E(G)} \ge (q-1)n$ and we are done.
So assume $2q-3 \ge \delta(G) \ge q$ and let $a=\delta(G)-(q-1).$
Let $v$ be a vertex with $\deg(v)=\delta(G).$
Since for every $u \in N[v]$, $uv$ is in a $K_q$, $\delta(G[N[v]])\ge q-1$.
Take a $K_q$ containing $v$ and let $A$ be the set with the $a$ other vertices of $N[v]$.
The number of edges in $G[N[v]]$ containing at least one vertex in $A$ is equal to
$$\sum_{u \in A} \deg_{G[N[v]]}(u) - \abs{E[A]} \ge (q-1)a-\binom{a}{2}.$$
Hence the number of edges in $G[N[v]]$ is at least $\binom{q}{2}+(q-1)a-\binom{a}{2}.$
Every vertex $z\in V \backslash N[v]$ has at least $q-2$ neighbors in $N(v)$ and at least $a+1$ additional neighbors.
This implies that 
\begin{align*}
    \abs{E(G)}&\ge  \binom{q}{2}+(q-1)a-\binom{a}{2}+(n-a-q)\left(q-1+\frac{a-1}{2}\right)\\ 
    &= n(q-1)-\binom{q}{2}+(n-2a-q)\frac{a-1}{2}\\
    &\ge n(q-1)-\binom{q}{2}. 
\end{align*}
\end{proof}

\begin{remark}
    The bound in Proposition~\ref{prop:deltaGeQcase} does not hold if one relaxes the condition $G[N(u) \cap N(v)]$ contains a $K_{q-2}$ to the condition $\abs{N(u) \cap N(v)}\geq q-2$, as was the case with Proposition~\ref{prop:delta=Qcase}.
Since  one can take a graph $(K_{q-2}\backslash M)+\left(\frac{n-(q-2)}{3}K_3\right)$\footnote{this is the graph join of $K_{q-2}$ and $\frac{n-(q-2)}{3}K_3$}  which is the graph that contains a clique of size $q-2$ minus a maximal matching $K_{q-2}\backslash M$ and  $\frac{n-(q-2)}{3}$ disjoint copies of $K_3$, such that every vertex of each $K_3$ is adjacent to the $q-2$ vertices of the clique minus a matching. This  construction has fewer edges than in Proposition~\ref{prop:deltaGeQcase} for every $q \ge 6$ and for every $n\geq q+1$ congruent to $q-2$ modulo $3$.

\end{remark}

\subsection{The minimum size of a $2$-self-centered graph $G$ with the triangle-property}

\begin{thr}\label{thr:mge2n-3}
     Let $G$ be an $n$-vertex graph satisfying the triangle-property and $\diam(G)=\rad(G)=2.$ Then the number of edges of $G$ is at least $2n-3.$
\end{thr}

\begin{proof}
    First suppose that $G$ has a vertex $v$ of degree $n-2$ i.e., $N[v]=V\backslash \{u\}$ for some vertex $u$ distinct from $v$.
    Let $c_v$ be the number of components in $G[N(v)]$.
    Since $G$ has diameter~$2$ and has the triangle-property, $u$ has at least two neighbors in every component and thus 
    $\abs{E(G)} \ge (n-2)+ (n-2-c_v) +2c_v \ge 2n-3.$
    

    
    Now assume that $G$ has the minimum number of vertices 
for which the statement of Theorem~\ref{thr:mge2n-3} does not hold. Thus $G$ has no vertex of degree $n-1.$
    We first observe that  the minimum degree of $G$ is  at least~$3.$
    
    \begin{claim}\label{clm:mindeg3}
        We have $\delta(G)\geq 3$
    \end{claim}
    \begin{claimproof}
        Since $G$ has the triangle-property, $\delta(G)\ge 2.$ Suppose by way of contradiction that there is a vertex  $u$ in $G$  of degree $2,$ and let  $N(u)=\{a,b\}.$ By the triangle-property $ab$  is an edge of $G$. 
        If the edge $ab$ belongs to a triangle different from $abu,$ then $G \backslash u$ is a smaller graph with a diameter and radius equal to $2$ which has the triangle-property, $n-1$ vertices and less than $2(n-1)-3$ edges, contradicting the minimality of $G$.
        
        If $N(a) \cap N(b)=\{u\}$, then let $A=N(a) \backslash \{u,b\}$ and $B=N(b) \backslash \{u,a\}.$
        Let $G[A]$ and $G[B]$ have $c_a$ and $c_b$ components respectively. 
        Note here that $\delta(G[a]), \delta(G[b]) \ge 1$ since $G$ has the triangle-property and thus every edge between $a$ and $A$ belongs to a triangle with an edge in $A$.
        Since $\diam(G)=2,$ there is an edge from every component of $A$ to every component of $B$, even more, since every edge is in a triangle there are at least $2$ edges between each such pair of components. Note that since 
        Hence the size of $G$ is at least
        \[
        3+\abs{A} + \abs{B} + (\abs{A}-c_a)+(\abs B -c_b) + 
        2c_ac_b =2n-3 + (2c_ac_b-c_a-c_b) \ge 2n-3.\qedhere
        \]
    \end{claimproof}

    Since $G$ has a diameter equal to two and has a triangle-property  for every two vertices of $G$ there is a vertex in the common neighborhood. Thus by Proposition~\ref{prop:deltaGeQcase} for $q=3$ and $n \ge 5$ we are done since  $\delta(G)=3$. 
    For  $n \le 5$ there are no graphs satisfying conditions of Theorem~\ref{thr:mge2n-3}.
\end{proof}

\section{Conclusion}\label{sec:conc}

In this paper, we determined the largest non-trivial family $\F \subseteq \binom{[n]}{k}$ which is union-$q$-efficient for $n=qk.$
Due to its similarities with the Erd\H{o}s matching conjecture~\cite{Erdos65}, one may wonder about the largest non-trivial family $\F \subseteq \binom{[n]}k$ without $q$ pairwise disjoint sets in the regime where the trivial family $\binom{[kq-1]}{k}$ (see~\cite{Frankl17}) is extremal.

The question of Tuza (Question~\ref{ques:Tuza}), asking for the largest non-trivial family $\F \subseteq \binom{[n]}{k}$ which is union-$q$-efficient, is still widely open when $k+q \le n<qk$. Here the case $n=q+k$ is equivalent with a hypergraph-Turán problem (Proposition~\ref{prop:equi_HT}).
The same question for non-trivial families $\F \subseteq \binom{[n]}{k}$ without $q$-simplex is equally natural and interesting.

The largest non-trivial $d$-wise intersecting families have been determined by O'Neill and Verstra\"ete~\cite{OV21}, proving a conjecture of Hilton and Milner~\cite{HM67}.
Analogous to some other results about the equivalence of the largest families,
one may expect that these constructions are also the largest non-trivial families which do not contain a $(d-1)$-simplex?

\section*{Acknowledgement}

Thanks to Peter Dankelmann, for pointing towards reference~\cite{ER62}.

\bibliographystyle{abbrv}
\bibliography{eff_fam}

\end{document}